\documentclass{amsart}
\usepackage{amsmath}
\usepackage{paralist}
\usepackage{url}
\usepackage{mathrsfs}
\usepackage{color}
\usepackage{graphicx}
\usepackage{caption}
\usepackage{subcaption}
\usepackage{multirow}
\usepackage{mathtools}
\usepackage{amssymb}
\usepackage{booktabs}
\usepackage{makecell}
\usepackage{pgfplots}
\pgfplotsset{compat=newest}
\usepackage{mathrsfs}
\usetikzlibrary{arrows}
\usepackage{xspace}
\usepackage{comment}
\usepackage{arydshln}
\usepackage{caption}

\newcommand{\re}{\mathbb{R}}

\newcommand{\N}{\mathbb{N}}

\newcommand{\mc}[1]{\mathcal{#1}}

\newcommand{\st}{\mathit{s.t.}}

\newcommand{\qmod}[1]{\mbox{QM}[#1]}

\newcommand{\RN}[1]{%
  \textup{\uppercase\expandafter{\romannumeral#1}}%
}

\newcommand{\cone}[1]{\mathit{cone}(#1)}

\newcommand{\bdes}{\begin{description}}
	\newcommand{\edes}{\end{description}}
\newcommand{\bal}{\begin{align}}
\newcommand{\eal}{\end{align}}
\newcommand{\bnum}{\begin{enumerate}}
	\newcommand{\enum}{\end{enumerate}}
\newcommand{\bit}{\begin{itemize}}
	\newcommand{\eit}{\end{itemize}}
\newcommand{\bea}{\begin{eqnarray}}
\newcommand{\eea}{\end{eqnarray}}
\newcommand{\be}{\begin{equation}}
\newcommand{\ee}{\end{equation}}
\newcommand{\baray}{\begin{array}}
	\newcommand{\earay}{\end{array}}
\newcommand{\bsry}{\begin{subarray}}
	\newcommand{\esry}{\end{subarray}}
\newcommand{\bca}{\begin{cases}}
	\newcommand{\eca}{\end{cases}}
\newcommand{\bcen}{\begin{center}}
	\newcommand{\ecen}{\end{center}}
\newcommand{\bbm}{\begin{bmatrix}}
	\newcommand{\ebm}{\end{bmatrix}}
\newcommand{\bmx}{\begin{matrix}}
	\newcommand{\emx}{\end{matrix}}
\newcommand{\bpm}{\begin{bmatrix}}
	\newcommand{\epm}{\end{bmatrix}}
\newcommand{\btab}{\begin{tabular}}
	\newcommand{\etab}{\end{tabular}}

\newtheorem{theorem}{Theorem}[section]

\newtheorem{lemma}[theorem]{Lemma}

\theoremstyle{definition}
\newtheorem{example}[theorem]{Example}

\newtheorem{alg}[theorem]{Algorithm}

\newtheorem{assumption}[theorem]{Assumption}

\numberwithin{equation}{section}

\pgfplotsset{every axis/.append style={tick label style={/pgf/number format/fixed},font=\scriptsize,ylabel near ticks,xlabel near ticks,grid=major}}

\pgfmathdeclarefunction{sumexp}{3}{%
  \begingroup%
  \pgfkeys{/pgf/fpu}
  \pgfmathsetmacro{\myx}{#1}%
  \pgfmathtruncatemacro{\myxmin}{#2}%
  \pgfmathtruncatemacro{\myxmax}{#3}%
  \pgfmathsetmacro{\mysum}{0}%
  \pgfplotsforeachungrouped\XX in {\myxmin,...,\myxmax}%
    {\pgfmathsetmacro{\mysum}{\mysum+exp(\XX)}}%
  \pgfmathparse{\mysum+exp(#1)}%
  \pgfmathfloattofixed\pgfmathresult
  \pgfmathsmuggle\pgfmathresult\endgroup%
}%

\begin{document}

\title[DRSR model with moment ambiguity]
{Distributionally robust shortfall risk portfolio model with moment ambiguity sets}

\author[Yi Yang]{Yi Yang}
\author[Liu Yang]{Liu Yang}
\author[Suhan Zhong]{Suhan~Zhong}

\address{Yi Yang, School of Mathematics and Computational Sciences, 
	Xiangtan University, Xiangtan, Hunan, China, 411105; 
	School of Science, Hunan City University, Yiyang, Hunan, China, 413000.}

\address{Liu Yang, School of Mathematics and Computational Sciences, 
	Xiangtan University, Xiangtan, Hunan, China, 411105; 
	Hunan Research Center of the Basic Discipline Fundamental Algorithmic Theory and Novel Computational Methods, Xiangtan University, 
	Xiangtan, Hunan, China, 411105}

\address{Suhan Zhong, School of Mathematical Sciences,
Shanghai Jiao Tong University,
Shanghai, China, 200240.}

\begin{abstract}
This paper employs shortfall risk to measure portfolio risk.
Assume stock returns follow polynomial relations with their influence factors.
We propose a moment-based distributionally robust optimization (DRO) shortfall 
risk portfolio model.
For piecewise linear loss functions, we show that this DRO model can be 
transformed into a tractable linear conic optimization problem with 
nonnegative polynomial cones. 
A Moment-SOS relaxation algorithm is proposed to solve the transformed problem.
Its finite and asymptotic convergence properties are studied.
For optimizers computed from our algorithm, 
we give convenient conditions verify their global optimality for the 
original DRO problem. 
Numerical experiments involving real stock market data are given to show 
the efficiency of our approach.	
\end{abstract}

\maketitle

\section{Introduction}

Portfolio theory remains a cornerstone of financial mathematics.
Extensive research has been conducted to find appropriate risk measurements
that balance theoretical rigor and practical applicability.
Markowitz's mean-variance model \cite{Markowitz1952} 
assumes normally distributed returns and symmetrically penalizes 
both upside and downside deviations.
Downside risk metrics like lower semi-variance were developed afterwards,
but they do not adequately account for loss magnitude.
In 1996, Jorion \cite{Jorion96} introduced Value-at-Risk (VaR) to quantify 
the largest expected loss at a given confidence level.
Based on that, Rockafellar and Uryasev \cite{RockUry00} subsequently 
introduced Conditional Value-at-Risk (CVaR) as a coherent risk measure.
It has distinct advantages compared to VaR and has been widely used in
portfolio management applications. 
However, CVaR still has computational limitations
with respect to consistency under randomization 
and heavy-tailed losses; see as in \cite{GuoXu19}.
On the other hand, the utility-based shortfall risk measure proposed by 
F\"{o}llmer and Schied \cite{FollmerSch02} is known for its effectiveness
in quantifying extreme risks associated with heavy-tailed losses. 
In this paper, we choose shortfall risk measure to study portfolio
optimization problems.

Consider a financial position modeled as a random variable 
$Z: (\Omega, \mathscr{F}, \mu) \rightarrow \mathbb{R}$,
where $(\Omega, \mathscr{F}, \mu)$ is the probability space.
The shortfall risk of $Z$ is defined as the optimal value of 
the following optimization problem:
\begin{equation}
	\label{eq:shortfall risk}
	\left\{\begin{array}{cl}
		\min \limits_{t\in \mathbb{R}} & t\\
		\st & \mathbb{E}_{\mu}[l(-Z-t)]\leq \lambda.
	\end{array}
	\right.
\end{equation}
Here $\mathbb{E}_{\mu}[\cdot]$ is the expectation taken under the measure $\mu$,
$l: \mathbb{R} \rightarrow \mathbb{R}$ is a convex, increasing and non-constant
loss function, and $\lambda$ is a prescribed confidence level. 
In applications, loss functions are often selected as piecewise linear functions, e.g.,
\begin{equation}\label{eq:loss}
	l(Z) \coloneqq \sup\limits_{1\le j\le m} \{a_jZ+b_j\},
\end{equation}
where $a_j,b_j$ are given real scalars.
We refer to \cite{Delage22,FollmerSch02,Giesecke2008,Wiesemann2014}
for more information about loss functions. 
Shortfall risk is commonly used to model portfolio problems \cite{Delage22,Gundel08}. 
A utility maximization problem with shortfall risk constraints is studied in \cite{Gundel08}.
An expected utility maximization problem under a utility-based shortfall constraint
is studied in \cite{JankeLi16}.
A preference robust multivariate utility-based shortfall risk measure
is studied in \cite{ZhangXu20}.

Consider a portfolio comprising \(n\) assets.
Let \(x = (x_1,\ldots, x_n)^{\top}\) denote the vector of investment proportions
constrained in 
\begin{equation}\label{eq:setX}
	X \coloneqq \{x\in\re^n: x\ge 0,\, x_1+\cdots+x_n=1\}.
\end{equation}
Let $r = (r_1,\ldots, r_n)^{\top}$ denote the vector of stock returns within a time period.
The financial position of the portfolio is given by 
\[
Z = x^{\top}r.
\]
To compute the shortfall risk of $Z$, a classic approach is to treat 
$r$ as a random vector by itself, see in \cite{GuoXu19,Gupte2023,Hu2018}. 
Another approach is to model $r$ as a vector-valued function by some influence factors.
This is because real world stock returns are known to be driven by influence factors 
such as macroeconomic conditions, market sentiment, and company fundamentals. 
Traditional models usually assume stock returns are linearly dependent on these 
factors, including Fama-French three-factor model \cite{Fama} and 
Carhart's four-factor model \cite{Carhart}.
Interestingly, recent studies suggest that modeling nonlinear dependencies 
between stock returns and their influence factors could offer a more realistic
representation of actual market behavior \cite{Guidolin,Kanas01,Thawornwong}. 
Polynomials are convenient functions that are often used to approximate
nonlinear relations in portfolio selection problems \cite{Fernadez22,ZhuEs23}. 
There are extensive studies on polynomial optimization \cite{Lass01,Lau09,Niebook}.
This motivates us to model the dependence of stock returns on 
their influence factors via polynomial functions.

Let $\xi = (\xi_1,\ldots, \xi_p)^{\top}$ denote the random vector of influence factors
that follows the distribution of a probability measure $\mu$ supported on
a given set $S$.
Assume each \(r_{i}:\mathbb{R}^p\to\re\) is a polynomial function in $\xi$.
For the loss function $l$ given in \eqref{eq:loss} and 
the feasible set $X$ given in \eqref{eq:setX},
we study the shortfall risk portfolio model in form of the following 
distributionally robust optimization (DRO) problem.
\begin{equation}\label{eq:DRO shortfall risk}
	\left\{\begin{array}{cl}
		\min\limits_{(t,x)} & t\\
		\st &\sup\limits_{\mu\in\mathcal{M}} 
		\mathbb{E}_{\mu} [l(-x^{\top} r(\xi)-t)]\leq\lambda,\\
		& x\in X,\, t\in\re.
	\end{array}
	\right.
\end{equation}
In the above, the {\it ambiguity set} $\mathcal{M}$ is a family of probability 
measures supported on $S$, which is used to describe the uncertainty.
Ambiguity sets are typically constructed by truncated moment information
\cite{ChenSim19,DelageYe10,LasWei21,NieYang23,NieZhong23},
or take the form of distributional neighborhood around the empirical measure
via metrics such as the Wasserstein distance \cite{Esfahani18,GaoKley22} or $\phi$-divergence
\cite{BenDen13,ZhaoLi24}.
It is common to assume the true probability measure is contained in $\mc{M}$.
In this paper, we choose the moment-based ambiguity set
\begin{equation}\label{eq:ambi}
	\mathcal{M} \coloneqq \big\{
	\mu \in \mathcal{P}(S) : \mathbb{E}_{\mu}\left([\xi]_{d}\right) \in \mc{B}
	\big\}.
\end{equation}
Here $\mathcal{P}(S)$ denotes the set of probability measures supported on 
\( S \subseteq \mathbb{R}^p \), 
\begin{equation}\label{eq:vec_xi}
	[\xi]_d \coloneqq \begin{bmatrix} 
		1 & \xi_1 & \cdots & \xi_p & \xi_1^2 & \xi_1 \xi_2 & 
		\cdots & \xi_p^d \end{bmatrix}^{\top}
\end{equation}
is the monomial vector of $\xi$ up to degree $d$, and
\(\mc{B}\) is a moment feasible set given by linear, 
second-order or semidefinite constraints.
Distributionally robust optimization is a universal framework 
in stochastic optimization.
Polynomial distributionally robust optimization is studied in \cite{LasWei21,NieYang23,NieZhong23,Rao}. 
Distributionally robust shortfall risk models are 
studied in \cite{GuoXu19,ZhangXu20}.
However, to the best of our knowledge, 
only few shortfall risk models treat stock returns 
	as nonlinear functions of their influence factors.
	In addition, DRO problems defined with piecewise polynomials 
	have not been explored in \cite{LasWei21,NieYang23,NieZhong23}.
	This motivates to consider the polynomial distributionally robust shortfall
	risk portfolio model with piecewise linear loss functions.

\subsection*{Contribution}
In this paper, we propose a distributionally robust shortfall risk model to 
investigate portfolio management problems. 
Assume that stock returns are polynomial functions of their influence factors, 
which are not required to be convex. 
We model these influence factors as random variables and describe 
their uncertain distributions by moment-based ambiguity sets 
as in \eqref{eq:ambi}.
Suppose the loss function is piecewise linear in the decision variables.
We show that such a DRO model can be equivalently transformed into 
a dual pair of linear conic optimization problem with moment and
nonnegative polynomial cones. 
A semidefinite algorithm is proposed to solve the reformulated optimization,
of which the finite and asymptotic convergence properties are studied. 
Convenient conditions are given to verify the global optimality of the original 
distributionally robust shortfall risk problem.
In particular, our algorithm can recover the probability measure 
that satisfies the worst-case expectation constraint.
Our main contributions can be summarized as follows.
\begin{itemize}
	
	\item 
	We establish a factor-inclusive distributionally robust
	shortfall risk model with moment-based ambiguity sets.
	By assuming stock returns follow polynomial relations 
	with their influence factors, 
	our model is more flexible to represent the complex stock market 
	compared to linear models. 
	
	\item  
	For piecewise linear loss functions, 
	we show that our distributionally robust shortfall risk model can be
	transformed into linear conic optimization problems with moment and 
	nonnegative polynomial cones.

	\item 
	We propose a semidefinite algorithm to solve the transformed linear conic 
	optimization and study its finite and asymptotic convergence properties. 
	Numerical experiments are given show the effectiveness of our DRO model 
	and the efficiency of our algorithm.
	
\end{itemize}

The rest of the paper is organized as follows. 
Preliminaries are given in Section~\ref{sc:pre}.  
In Section~\ref{sc:trans}, we introduce the transformation of the distributionally 
robust shortfall risk problem into solvable linear conic optimization problems.
A Moment-SOS approach for solving the distributionally robust shortfall risk problem 
is summarized in Section~\ref{sc:momsos}.
Numerical examples and applications are presented in Section~\ref{sc:numerical}. 
The conclusions are given in Section~\ref{sc:con}.

\section{Preliminaries}\label{sc:pre}

\subsection*{Notation}
The symbol $\mathbb{N}$ denotes the set of nonnegative integers.
The symbol $\re$ (resp. $\re_+$) denotes the set of real 
(resp. nonnegative real) numbers. 
The $\N^n$ (resp. $\re^n$, $\re_+^n$) denotes the $n$-dimensional vector 
with entries in $\N$ (resp. $\re$, $\re_+$).
For $t\in\mathbb{R}$, $\lceil t\rceil$ denotes the smallest 
integer no less than $t$.
For a positive integer $d$, $[d]\coloneqq \{1,\ldots, d\}$. 
Let $x=(x_1,\ldots,x_n)^{\top}$ be a vector of variables.
The $\mathbb{R}[x]$ denotes the ring of real polynomials in $x$,
and $\re[x]_d$ denotes the subset of polynomials with degree at most $d$. 
For $f\in \re[x]$, its degree is denoted by $\deg(f)$. 
For a tuple of polynomials $g = (g_1,\ldots, g_m)^{\top}$, 
we use $\deg(g)$ to denote the highest degree of $g_i$ for all $i\in [m]$.
We use the symbol $e\coloneqq (1,\ldots, 1)^{\top}$ to denote the all-one vector,
and use $e_i$ to represent the $i$th standard basis vector in $\mathbb{R}^n$.
A symmetric matrix $A\in \mathbb{R}^{n\times n}$ is called positive 
semidefinite (psd) if $x^{\top}Ax\ge 0$ for all $x\in \mathbb{R}^n$.
The notation $A\succeq 0$ denotes that $A$ is psd.
Given $S\subseteq\re^p$, we use $\mathcal{P}(S)$ to denote the set of 
all probability measures supported on $S$, and $\mathcal{M}(S)$  for the set of all finite nonnegative measures supported on $S$.
For a given power vector $\alpha = (\alpha_1,\ldots,\alpha_n)^{\top}\in\mathbb{N}^n$, 
we denote the monomial
\[ 
x^{\alpha} \coloneqq x_1^{\alpha_1}\cdots x_n^{\alpha_n}\quad
\mbox{with}\quad |\alpha| \coloneqq \alpha_1+\cdots+\alpha_n. 
\]
For a degree $d$, denote the power set 
\[
\mathbb{N}_d^n\coloneqq \{\alpha\in\mathbb{N}^n: |\alpha|\le d\}
\]
and the monomial vector 
$[x]_d := (1,\, x_1,\, \cdots,\, x_n,\, x_1^2,\, \cdots,\, x_n^d)^{\top}$. 
Let $V$ be an inner product space equipped with the inner product 
$\langle \cdot, \cdot\rangle$,
whose dual space is denoted by $V^*$.
For a nonempty subset $S\subseteq V$, we use
\[
cone(S) \coloneqq \Big\{v\in V: v = \sum\limits_{i=1}^N \lambda_iv_i,\, 
v_i\in S,\, \lambda_i\ge 0,\, N\in\N\Big\}
\]
to denote the conic hull generated by $S$.
The dual cone of $S$ is 
\[
S^* \coloneqq \{u\in V^*: \langle u,v\rangle \ge 0\quad \forall v\in S\}.
\]
It is clear that $S^* = (cone(S))^*$.

\subsection{Nonnegative polynomials}
\label{ssc:PO}
A polynomial $\sigma\in \mathbb{R}[x]$ is said to be sum-of-squares (SOS) 
if it can be decomposed as 
\[
\sigma = \sigma_{1}^{2}+\cdots+\sigma_{k}^{2}\quad 
\mbox{for some}\quad
\sigma_1,\ldots,\sigma_k\in\mathbb{R}[x].
\]
We use $\Sigma[x]$ to denote the set of all SOS polynomials and 
use $\Sigma[x]_d\coloneqq \Sigma[x] \cap \mathbb{R}[x]_{d}$ 
to denote its $d$th order truncation.
Let $g =(g_1,\ldots,g_m)^{\top}$ be a polynomial vector in $x$.
The quadratic module of $g$ is defined by
\[ 
\qmod{g} \coloneqq \Sigma[x]+g_{1}\cdot\Sigma[x]+\cdots+g_{m} \cdot \Sigma[x]. 
\]
For $k\ge \lceil \deg(g)/2\rceil$, the $k$th order truncated quadratic module of $g$
is defined by
\[
\qmod{g}_k \coloneqq \Sigma[x]_{2 k}+g_{1} \cdot \Sigma[x]_{2k-\deg(g_1)}
+ \cdots+ g_{m} \cdot \Sigma[x]_{2k-\deg(g_m)}.
\]
Clearly, $\qmod{g}_k\subseteq \qmod{g}_{k+1}\subseteq \qmod{g}$ for each $k$. 
Let 
\[
K=\{ x\in\mathbb{R}^n: g(x)\ge 0 \}.
\]
The cone of nonnegative polynomials over $K$, denoted by $\mathscr{P}(K)$, 
is given by
\[ 
\mathscr{P}(K) \coloneqq \{ q\in \mathbb{R}[x]: q(x)\ge 0
\quad \forall x\in K \}. 
\]
Its $d$th degree truncation is denoted by 
$\mathscr{P}_d(K) \coloneqq \mathscr{P}(K)\cap \mathbb{R}[x]_d$.
If $q\in \qmod{g}$, then it is easy to verify that $q\in\mathscr{P}(K)$.
The $\qmod{g}$ is said to be {\it archimedean} if there exits 
$q\in \qmod{g}$ such that $q(x)\ge 0$ determines a compact set.
Suppose $\qmod{g}$ is archimedean. If a polynomial $f$ is positive on $K$, 
then $f\in \qmod{g}$. This conclusion is often referenced as 
{\it Putinar's Positivstellensatz} \cite{Putinar}.

\subsection{Truncated moment problems}
\label{ssc:tms}
Let $\xi = (\xi_1,\ldots, \xi_p)^{\top}$. 
A real vector $y = (y_{\alpha})$ labelled by $\alpha\in \N_d^p$ 
is called a truncated multi-sequence (tms) in $\xi$ of degree $d$.
A tms with respect to $x$ is defined similarly.
Let $\mathbb{R}^{\mathbb{N}_{d}^{p}}$ denote the collection of all tms 
in $\xi$ of degree $d$. For any tms $y\in \re^{\N_d^p}$,
	we define the linear Riesz functional $\mathscr{L}_y: \re[\xi]_d\to \re$
	by setting $\mathscr{L}_y(\xi^{\alpha}) = y_{\alpha}$ 
	for all $\alpha\in \N_d^p$ and extending linearly to $\re[\xi]_d$.
	This induces the following bilinear operation
	\begin{equation}\label{eq:<f,y>}
		\langle f, y\rangle:=\mathscr{L}_y(f)
	\end{equation}
	for all $y\in \mathbb{R}^{\mathbb{N}_{d}^{p}}$ and $f\in\re[\xi]_d$.
A tms $y\in\mathbb{R}^{\mathbb{N}_d^p}$ is said to admit a measure on $S$
if there exists a measure $\mu$ such that 
$y_{\alpha} = \int \xi^{\alpha}\mathtt{d}\mu$ 
for all $\alpha\in\mathbb{N}_d^p$.
Such a measure $\mu$ is called an $S$-representing measure of $y$.
The problem for checking whether or not a tms admits a measure is called the 
{\it truncated moment problem}.
\begin{example}
	Let $p = d =2$. 
	The tms $y=(y_0, y_{10}, y_{01}, y_{20}, y_{11}, y_{02})^{\top}$ admits the 
	Dirac measure $\delta_u$ supported at $u = (u_1,u_2)^{\top}$ if and only if 
	$y = (1, u_1, u_2, u_1^2, u_1u_2, u_2^2)^{\top}.$
	In particular, the zero tms $y=0$ admits the identically zero measure.
\end{example}
For a tms $y$, we use $\mbox{meas}(y, S)$ to denote the set of 
$S$-measures admitted by $y$.
Denote the set of $d$-th order tms that admit a measure supported on $S$ by
\[
\mathscr{R}_d(S) \coloneqq \{z \in \mathbb{R}^{\mathbb{N}_d^p} 
:\operatorname{meas}(z, S) \neq \emptyset\} .
\]
The $\mathscr{R}_d(S)$ is a convex cone, which is also closed 
when $S$ is compact. 
Interestingly, it can be equivalently expressed as 
	\[
	\mathscr{R}_d(S)=\operatorname{cone}
	\left(\left\{[\xi]_d: \xi \in S\right\}\right).
	\]
	This is because every $y\in\mathscr{R}_d(S)$ must also 
	admits a finitely atomic representing measure $\nu$ supported on $S$
	\cite[Theorem~2.7.1]{Niebook}.
	Based on the inner product defined in \eqref{eq:<f,y>}, 
	the dual relation holds that 
	\begin{equation}\label{eq:PDdual}
		\mathscr{R}_d(S)^* = \mathscr{P}_d(S),\quad
		\mathscr{P}_d(S)^* = \overline{\mathscr{R}_d(S)},
	\end{equation}
	where $\overline{\mathscr{R}_d(S)}$ denotes the closure of $\mathscr{R}_d(S)$
	\cite[Theorem~2.4.7]{Niebook}.	
	We also refer to  \cite{Nie15} and \cite{NieYang23} for 
	more details about this result.
For a generic semialgebriac set $S$, the moment cone $\mathscr{R}_d(S)$
is difficult to characterize in computations. 
A common approach is to approximate it by moment and localizing matrices.

\subsection{Moment and localizing matrices}
Let $q\in\mathbb{R}[x]_{2t}$. 
For $k\ge t$ and a tms $y\in\mathbb{R}^{\mathbb{N}_{2k}^n}$, 
there exists a symmetric matrix $L_q^{(k)}[y]$ such that 
\begin{equation}\label{eq:locmat}
	\operatorname{vec}(a)^{\top}(L_{q}^{(k)}[y])
	\operatorname{vec}(b) = \mathscr{L}_{y}(q a b),
	\quad \forall a,b\in\mathbb{R}[x]_{k-t},
\end{equation}
where $\operatorname{vec}(a), \operatorname{vec}(b)$ 
denote the coefficient vectors of $a, b$ respectively. 
The $L_q^{(k)}[y]$ is called the $k$th {\it localizing matrix} of $q$. 
For the special case that $q=1$ is the constant polynomial, 
we define the moment matrix by
\begin{equation}\label{eq:mommat}
	M_{k}[y] \coloneqq L_{1}^{(k)}[y].
\end{equation}	
\begin{example}
	Let $n=2$, $q(x) = x_1-x_2^2$. We have
	\[
	M_1[y] = \bbm y_{00} & y_{10} & y_{01}\\
	y_{10} & y_{20} & y_{11}\\
	y_{01} & y_{11} & y_{02}\ebm,\quad
	L_q^{(2)}[y] = \bbm y_{10}-y_{02} & y_{20}-y_{12} & y_{11}-y_{03}\\
	y_{20}-y_{12} & y_{30}-y_{22} & y_{21}-y_{13}\\
	y_{11}-y_{03} & y_{21}-y_{13} & y_{12}-y_{04}\ebm.
	\]
\end{example}
Let $g = (g_1,\ldots,g_m)^{\top}$ be a polynomial vector in $x$.
For $k\ge \lceil \deg(g)/2\rceil$, 
define the tms cone
\begin{equation} \label{mom:S(g):2d}
	\mathscr{S}[g]_{k} \coloneqq
	\big\{
	y \in \mathbb{R}^{ \mathbb{N}^n_{2k} }:
	M_k[y] \succeq 0, \,  L_{g_1}^{(k)}[y] \succeq 0, \ldots,
	L_{g_m}^{(k)}[y] \succeq 0
	\big\}.
\end{equation}
Given $K = \{x\in\re^n: g(x)\ge 0\}$, a necessary condition for $y\in\mathscr{R}_{2k}(K)$
is that $y\in\mathscr{S}[g]_k$.
It can be verified that the following dual relations hold:
\begin{equation}\label{eq:qmod_dual}
	(\qmod{g}_k)^* = \mathscr{S}[g]_k,\quad
	(\mathscr{S}[g]_k)^* = \overline{\qmod{g}_k}.
\end{equation}
For more details about truncated moment problems, 
we refer to \cite{NieAtrunc}, \cite[Chapter~2.7]{Niebook} 
and references therein.

\section{Transformation of portfolio DRO problem}
\label{sc:trans}

In this section, we transform the distributionally 
robust shortfall risk problem 
\eqref{eq:DRO shortfall risk} into a linear cone optimization problem.  

\subsection{Reformulation of the worst-case expectation constraint}
Consider the moment ambiguity set $\mc{M}$ as in \eqref{eq:ambi}.
Since  $\mc{M}\subseteq \mc{P}(S)$ is a set of probability measures, 
we have $\mathbb{E}_{\mu}[\lambda] = \lambda$ for 
every $\mu\in \mc{M}$ and any prescribed confidence level $\lambda>0$.
Hence the expectation constraint
\[
\mathbb{E}_{\mu} [l(-x^{\top} r(\xi)-t)]\leq\lambda
\,\,\Leftrightarrow\,\,
\mathbb{E}_{\mu}[\lambda-l(-x^{\top}r(\xi)-t)]\geq 0.
\]
Recall that $l(Z) = \sup\{a_jZ+b_j: j\in[m]\}$. 
For $j = 1,\ldots, m$, denote
\begin{equation}\label{eq:f_j}
	f_j(x,t,\lambda,\xi) \coloneqq 
	\lambda - \big(a_j(-x^{\top}r(\xi)-t)+b_j\big).
\end{equation}
	The worst-case expectation constraint in \eqref{eq:DRO shortfall risk} 
	thus becomes
	\begin{equation}\label{eq:inf_ex_inf}
		\inf_{\mu \in \mathcal{M}} \mathbb{E}_\mu
		\Big[\inf _{j \in[m]} f_j(t,x,\lambda,\xi)\Big] \geq 0.
	\end{equation}
	\begin{lemma}\label{lem:equiv_cons} 
		For given $(t,x,\lambda)$, the constraint \eqref{eq:inf_ex_inf} holds if 
		and only if 
		\begin{equation}\label{eq:sum_ex_ge}
			\sum\limits_{j=1}^m \mathbb{E}_{\mu_j} [f_j(t,x,\lambda, \xi)]\ge 0
			\quad \mbox{for all}\quad
			\begin{array}{c}
				\mu_1,\ldots,  \mu_m\in \mc{M}(S)\,\,\mbox{with}\\
				\mu_1+\cdots+\mu_m\in \mc{M},
			\end{array}
		\end{equation}
		where $\mc{M}(S)$ is  the set of all finite nonnegative measures supported on $S$. In particular, if there exist measures 
		$\mu_1^*, \ldots, \mu_m^*\in\mc{M}(S)$
		with $\mu^* = \sum_{j=1}^m \mu_j^*\in\mc{M}$ such that
		\begin{equation}\label{eq:exp=infexp}
			\sum_{j=1}^m  \mathbb{E}_{\mu_j^*}[f_j(t,x,\lambda,\xi)]
			\,=\,
			\inf\limits_{\substack{\mu_1,\ldots,\mu_m\in\mc{M}(S)\\
					\mu_1+\cdots+\mu_m\in\mc{M}}}
			\sum\limits_{j=1}^m  \mathbb{E}_{\mu_j}[f_j(t,x,\lambda,\xi)] 
		\,=\, 0, 
	\end{equation}
	then $\mu^*$ is a worst-case measure of \eqref{eq:inf_ex_inf}
	that satisfies
	\[
	\mathbb{E}_{\mu^*}\Big[\inf\limits_{j\in[m]} f_j(t,x,\lambda,\xi)\Big] = 
	\inf\limits_{\mu\in \mc{M}}\mathbb{E}_{\mu}\Big[\inf\limits_{j\in[m]}
	f_j(t,x,\lambda,\xi)\Big] = 0.
	\]
\end{lemma}
\begin{proof}
	For convenience, we denote $f_j(\xi) = f_j(t,x,\lambda,\xi)$ 
	at given $(t,x,\lambda)$. Let
	\[\begin{array}{l}
		S_1\coloneqq \{\xi\in S: f_1(\xi) = \inf_{i\in[m]} f_i(\xi)\},\\
		S_j \coloneqq \left\{
		\xi \in S: f_j(\xi) = \inf_{i\in[m]} f_i(\xi),\,
		\xi\not\in S_1\cup\cdots\cup S_{j-1}
		\right\}
	\end{array}\]
	for $j=2,\ldots, m$. It is clear that $S_1,\ldots, S_m$ form a 
	disjoint decomposition of $S$ such that 
	$f_j(\xi) = \inf_{i\in[m]} f_i(\xi)$ for every $\xi\in S_j$.
	Let $\mu\in \mc{M}$ be an arbitrary measure.
	We use $\mu|_{S_j}$ to denote its restriction on $S_j$ for each $j$.
	It holds that
	\[
	\mathbb{E}_\mu\Big[\inf_{j \in[m]} f_j(\xi)\Big]
	= \sum_{j=1}^m \mathbb{E}_{\mu|_{S_j}}[f_j(\xi)] 
	\geq \inf_{\substack{\mu_1, \ldots, \mu_m \in \mc{M}(S) \\ 
			\mu_1+\cdots+\mu_m \in \mathcal{M}}}
	\sum_{j=1}^m\mathbb{E}_{\mu_j}\left[f_j(\xi)\right].
	\]
	Since the above holds for arbitrary $\mu\in \mc{M}$, 
	the condition \eqref{eq:sum_ex_ge} is sufficient 
	for the inequality \eqref{eq:inf_ex_inf} to hold.
	Conversely, for any $\mu_1,\ldots, \mu_m\in \mc{M}(S)$
	such that $\bar{\mu} = \sum_{j=1}^m \mu_j\in \mc{M}$,
	we have
	\[
	\inf\limits_{\mu\in \mc{M}} \mathbb{E}_{\mu}\Big[\inf\limits_{j\in[m]} f_j(\xi)\Big]
	\le \mathbb{E}_{\bar{\mu}}\Big[\inf\limits_{j\in[m]} f_j(\xi)\Big]
	= \sum\limits_{j=1}^m\mathbb{E}_{\mu_j}\Big[\inf\limits_{j\in[m]} f_j(\xi)\Big]
	\le \sum\limits_{j=1}^m\mathbb{E}_{\mu_j}[f_j(\xi)].
	\]
	This implies that \eqref{eq:sum_ex_ge} is also a necessary condition for
	\eqref{eq:inf_ex_inf} to hold.
	
	In particular, suppose the equality \eqref{eq:exp=infexp} is satisfied with 
	measures $\mu_1^*,\ldots, \mu_m^*\in \mc{M}(S)$ 
	and $\mu^* = \sum_{j=1}^m \mu_j^*\in\mc{M}$. Then
	\[
	0 = 
	\inf\limits_{\substack{\mu_1,\ldots,\mu_m\in \mc{M}(S)\\
			\mu_1+\cdots+\mu_m\in\mc{M}}}
	\sum\limits_{j=1}^m \mathbb{E}_{\mu_j}[f_j(\xi)]
	\le \mathbb{E}_{\mu^*}\Big[\inf\limits_{j\in[m]}f_j(\xi)\Big]
	\le \sum\limits_{j=1}^m \mathbb{E}_{\mu_j^*}[f_j(\xi)]
	= 0.
	\] 
	This implies that $\mu^*$ is a worst-case measure of \eqref{eq:inf_ex_inf}.
\end{proof}
By Lemma~\ref{lem:equiv_cons}, the DRO problem \eqref{eq:DRO shortfall risk} 
is equivalent to
\begin{equation}\label{eq:SC_DRO shortfall risk}
	\left\{\begin{array}{rl}
		\min \limits_{(t, x)} & t \\
		\st &  t\in \re, \, x\in X,\\
		& \sum\limits_{j=1}^m 
		\mathbb{E}_{\mu_j}\left[ f_j(t,x,\lambda, \xi)\right] \geq 0\\
		& \forall \mu_j\in \mc{M}(S),\,
		\sum\limits_{j=1}^m\mu_j\in \mc{M},\,j\in[m].
	\end{array}\right.
\end{equation}
In the above, each polynomial 
$f_j(t,x,\lambda,\xi) = \lambda - (a_j(-x^{\top}r(\xi)-t)+b_j )$
is linear in the decision variables $(t,x)$. 
Recall that $d$ is the highest moment order used in the ambiguity 
set $\mc{M}$. For convenience, consider
$\deg(r) = d$. For cases that $\deg(r)\not=d$, 
one can make similar analysis with $\max\{\deg(r), d\}$.
For each $j$, there exist a unique matrix 
$A_j\in\mathbb{R}^{\binom{p+d}{d} \times n}$ 
and a vector $c_{j} \in \mathbb{R}^{\binom{p+d}{d}}$ such that
\begin{equation}\label{eq:Ajbj}
	f_j(t,x,\lambda,\xi) = (a_je_1\cdot t+A_jx+c_j)^{\top}[\xi]_{d},
\end{equation}
where $e_1 = (1,0,\cdots,0)^{\top}\in\re^{\binom{p+d}{d}}$ is the unit vector 
and $[\xi]_{d}$ is the monomial vector as in \eqref{eq:vec_xi}.
For any $\mu_j\in\mc{M}(S)$, it holds that
\[
\mathbb{E}_{\mu_j}[f_j(t,x,\lambda,\xi)]
=(a_je_1\cdot t+A_j x+c_j)^{\top}\mathbb{E}_{\mu_j}[ [\xi]_{d} ] .
\]
We use the following example to show how to construct such $A_j, c_j$. 
\begin{example}\label{ex:Ajbj}
Consider $n = 3, p =2$ and $r = (r_1,r_2,r_3)^{\top}$ with
\[
r(\xi) = \bbm 0.5+0.6\xi_{1}-\xi_{2}+2\xi_{1}^{2}+\xi_{1}\xi_{2}-\xi_{2}^{2}\\
-0.02+0.2\xi_{1}+3\xi_{2}-2\xi_{1}^{2}-2\xi_{1}\xi_{2}+\xi_{2}^{2}\\ 
0.3+0.3\xi_{1}-1.5\xi_{2}+\xi_{1}^{2}+\xi_{1}\xi_{2}+2\xi_{2}^{2}\ebm.
\]
For $a_j = b_j = 1$ and $\lambda = 0.6$, we obtain 
\[\begin{aligned}
	f_j(t,x,\lambda, \xi)
	& = a_j(x^{\top}r(\xi)+t)+(\lambda-b_j) \\
	& = t+ x^{\top}r(\xi) - 0.4\\
	& = \left(\bbm 1\\0\\0\\0\\0\\0 \ebm t
	+ \underbrace{
		\left[\begin{array}{rrr} 
			0.5 & -0.02 & 0.3\\
			0.6 & 0.2 & 0.3\\
			-1 & 3 & -1.5\\
			2 & -2 & 1 \\
			1 & -2 & 1\\
			-1 & 1 & 2\end{array}\right]}_{A_j}
	\bbm x_1\\x_2\\x_3\ebm +
	\underbrace{\bbm - 0.4\\0\\0\\0\\0\\0\ebm}_{c_j}\right)^{\top} 
	\underbrace{\bbm 1\\ \xi_1\\ \xi_2\\ \xi_1^2\\ \xi_1\xi_2\\ \xi_2^2\ebm}_{[\xi]_2}.
\end{aligned}\]
\end{example}

\subsection{Linear conic transformation}
As shown in Subsection~\ref{ssc:tms},
the cone of $d$-th order truncated moments that admit measures
supported on $S$ is
$\mathscr{R}_d(S)=\operatorname{cone}
\left(\left\{[\xi]_d: \xi \in S\right\}\right)$. Denote
\begin{equation}\label{eq:K}
	K\coloneqq cone\left(\{\mathbb{E}_{\mu}[ [\xi]_{d}]: \mu\in\mc{M}\}\right)
	= \mathscr{R}_d(S)\cap \operatorname{cone}(\mc{B}).
\end{equation}
With the reformulation \eqref{eq:Ajbj}, the expectation constraint
in \eqref{eq:SC_DRO shortfall risk} is satisfied if and only if 
the following optimization problem has a nonnegative optimal value.
\begin{equation}\label{eq:theta_min}
	\left\{\begin{array}{rl}
		\theta\coloneqq \min\limits_{(y_1,\ldots, y_m)} & 
		\sum\limits_{j=1}^m (a_je_1\cdot t+ A_jx+c_j)^{\top}y_j\\
		\st & y_j\in\mathscr{R}_d(S),\,j\in[m],\\
		& y_1+\cdots+y_m\in K.
	\end{array}\right.
\end{equation}
By Lagrangian duality theory and the dual relation \eqref{eq:PDdual}, 
$\theta\ge 0$ if and only if there exists $w^{\top}[\xi]_d\in K^*$ such that 
\[
(a_je_1\cdot t+A_jx+c_j-w)^{\top}[\xi]_d\in \mathscr{P}_d(S),\quad j=1,\ldots, m,
\]
where $K^*$ is the dual cone of $K$ induced by the bilinear operation \eqref{eq:<f,y>}. 
Thus, we obtain the linear conic reformulation 
of the DRO problem \eqref{eq:SC_DRO shortfall risk} as
\begin{equation}\label{eq:DRO_equiv}
	\left\{\begin{array}{rl}
		\min\limits_{(t,x, w)} & t \\
		\st & t\in \mathbb{R},\,w^{\top}[\xi]_d\in K^*,\\
		&  x\in X = \{x\in\re^n: e^{\top}x = 1,x\ge 0\},\\
		& (a_je_1\cdot t+ A_j x + c_j - w)^{\top}[\xi]_d \in \mathscr{P}_d(S),\quad \forall j\in[m].
	\end{array}
	\right.
\end{equation}
where $e = (1,\ldots, 1)^{\top}\in\re^n$ is the vector of all ones.
Since \eqref{eq:SC_DRO shortfall risk} is equivalent to \eqref{eq:DRO shortfall risk},
the above problem is also an equivalent linear conic optimization
reformulation of the original DRO problem.

Next, we construct the Lagrangian dual problem of \eqref{eq:DRO_equiv}. 
Let 
\[ 
\gamma_0\in\re,\quad 
\gamma = (\gamma_1,\ldots, \gamma_n)^{\top}\ge 0,\quad  
\bar{y}\in \overline{K},\quad
y_1,\ldots, y_m\in\overline{\mathscr{R}_d(S)},
\]
where $\overline{K}$ and $\overline{\mathscr{R}_d(S)}$ are closures of 
$K$ and $\mathscr{R}_d(S)$ respectively.
The Lagrange function for problem \eqref{eq:DRO_equiv} is 
\[
\begin{aligned}
	&\mathcal{L}\left(t,x,w; \gamma_0, \gamma,\bar{y}, y_1,\ldots, y_m\right)\\
	&= t - \gamma_0(e^\top x - 1) - \gamma^\top x - 
	\bar{y}^{\top} w  - \sum_{j=1}^m y_j^{\top} 
	(a_je_1\cdot t+ A_j x + c_j - w ) \\
	&= \gamma_0 - \sum_{j=1}^m y_j^{\top} c_j + t\Big(1-\sum\limits_{j=1}^m a_je_1^{\top}y_j\Big)
	-x^{\top}\Big(\gamma_0e+\gamma+\sum\limits_{j=1}^m A_j^{\top}y_j\Big)
	-w^{\top}\Big(\bar{y}-\sum\limits_{j=1}^m y_j\Big) .
\end{aligned}
\]
For given dual variables, 
the Lagrangian function is bounded from below for 
all $(t,x,w)$ in the whole space
if and only if 
\[
1-\sum\limits_{j=1}^m a_je_1^{\top}y_j = 0,\quad
\gamma_0e+\gamma+\sum\limits_{j=1}^m A_j^{\top}y_j = 0,\quad
\bar{y}-\sum\limits_{j=1}^m y_j = 0.
\] 
Consequently, the dual problem of \eqref{eq:DRO_equiv} is
\begin{equation}\label{eq:DRO_dual}		
	\left\{\begin{array}{cl}
		\max\limits_{(\gamma_0, \gamma, y_1,\ldots, y_m)}& 
		\gamma_0-\sum\limits_{j=1}^m c_j^{\top}y_j \\
		\st & \gamma_0\in\re,\, \gamma\in\re^n,\,
		\gamma\ge 0,\, y_j\in \overline{\mathscr{R}_d(S)},\,j\in[m],\\
		& \sum\limits_{j=1}^m a_j(e_1^{\top}y_j) = 1,\,
		\gamma_0e+\gamma+\sum\limits_{j=1}^m A_j^{\top}y_j = 0,\\
		& y_1+\cdots+y_m\in \overline{K}.
	\end{array}\right.
\end{equation}	
\begin{theorem}\label{thm:equiv}
	Suppose $(t^*, x^*,w^*)$ is an optimizer of \eqref{eq:DRO_equiv}
	and $(\gamma_0^*, \gamma^*, y_1^*,\ldots,y_m^*)$ is an optimizer of \eqref{eq:DRO_dual}.
	Assume $S$ is compact and $K=\overline{K}$ and the strong duality between \eqref{eq:DRO_equiv}--\eqref{eq:DRO_dual}.
	Then the problem \eqref{eq:DRO_dual} is equivalent 
	to \eqref{eq:DRO shortfall risk} 
	in the sense that they share the same optimal value and that
	$y^* = y_1^*+\cdots+y_m^*$ admits a nonzero multiple of
	a worst-case measure $\mu^*$ that satisfies
	\[
	\mathbb{E}_{\mu^*}\Big[\inf\limits_{j\in[m]} f_j(t^*,x^*,\lambda,\xi)\Big] = 
	\inf\limits_{\mu\in \mc{M}}\mathbb{E}_{\mu}\Big[\inf\limits_{j\in[m]}
	f_j(t^*,x^*,\lambda,\xi)\Big] = 0.
	\]	
\end{theorem}
\begin{proof}
	By previous analysis, the problem \eqref{eq:DRO_equiv} is 
	an equivalent reformulation of \eqref{eq:DRO shortfall risk}.
	By the strong duality,
	the problem \eqref{eq:DRO_dual} has the same optimal value with the original DRO.
	Moreover, by the feasibility of \eqref{eq:DRO_dual}, we get
	\[
	t^* = \gamma_0^*-\sum_{j=1}^m c_j^{\top}y_j^*+ 
	\Big(	1-\sum\limits_{j=1}^m a_je_1^{\top}y_j^*\Big)t^*
	- \Big(\gamma_0^*e+\gamma^* + \sum\limits_{j=1}^m A_j^{\top}y_j^*\Big)^{\top}x^*.
	\]
	After simplification with $e^{\top}x^* = 1$, it implies
	\[
	0\le \sum\limits_{j=1}^m (a_je_1\cdot t^*+A_jx^*+c_j)^{\top}y_j^*
	= -(\gamma^*)^{\top}x^*\le 0,
	\]
	where the right inequality holds by $\gamma^*\ge 0, x^*\ge 0$
	and the left inequality holds by 
	\[
	\sum\limits_{j=1}^m 
	\big\langle (a_je_1\cdot t^*+A_jx^*+c_j-w^*)^{\top}[\xi]_d,y_j^*\big\rangle
	+ \Big\langle (w^*)^{\top}[\xi]_d, \sum\limits_{j=1}^m y_j^*\Big\rangle \ge 0.
	\]
	Consequently, $\sum_{j=1}^m (a_je_1\cdot t^*+A_jx^*+c_j)^{\top}y_j^* = 0$.
	Under the assumption that $S$ is compact, $\mathscr{R}_d(S)$ is closed 
	by \cite[Theorem~2.4.7]{Niebook}, thus each $y_j^*$ uniquely admits a 
	finitely atomic measure $\mu_j^*$ supported on $S$.
	In addition, $\sum_{j=1}^m a_j(e_1^{\top}y_j)=1$ and $\overline{K} = K$ 
	ensure that $y^* = \sum_{j=1}^m y_j^*$ admits a nonzero multiple 
	of the probability measure 
	$\mu^* = (\sum_{j=1}^m\mu_j^*)/(\sum_{j=1}^m e_1^{\top}y_j^*)\in \mc{M}$.
	It is easy to verify that 
	\[
	0\le \mathbb{E}_{\mu^*}\Big[\inf\limits_{j\in[m]} f_j(t^*,x^*,\lambda,\xi)\Big]
	\le \Big(1/\sum\limits_{j=1}^m e_1^{\top}y_j^*\Big)
	\sum\limits_{j=1}^m \mathbb{E}_{\mu_j^*}[f_j(t^*,x^*, \lambda,\xi)] = 0,
	\]
	where the left inequality holds since $t^*$ is the minimum of \eqref{eq:DRO shortfall risk}.
	Therefore, $\mu^*$ is a worst-case measure.
	\end{proof}

	\subsection{Moment-SOS relaxations}\label{ssc:Kstruc}
To solve the dual pair \eqref{eq:DRO_equiv}--\eqref{eq:DRO_dual}, 
we need to find convenient expressions for 
convex cones $K^*$, $\mathscr{P}_d(S)$ and $\overline{\mathscr{R}_d(S)}$, $\overline{K}$.
When $\xi$ is multivariable, $\mathscr{P}_d(S)$ and $\overline{\mathscr{R}_d(S)}$
are usually approximated by semidefinite constraints.
Recall that $\mc{B}\subseteq \re^{\N_d^p}$ is given by linear, second-order
or semidefinite constraints and that 
$K = \mathscr{R}_d(S)\cap \operatorname{cone}(\mc{B})$.
As shown in \cite{NieYang23,NieZhong23}, 
if $S, \mc{B}$ are compact and $K$ has a nonempty interior, 
we have 
\[
\mathscr{R}_d(S) = \overline{\mathscr{R}_d(S)},\quad
K = \overline{K},\quad
K^* = \big(\mathscr{R}_d(S)\cap \operatorname{cone}(\mc{B})\big)^* = \mathscr{P}_d(S)+\mc{B}^*.
\]
In particular, if $\xi$ is univariate, $K$ and its dual can be explicitly 
expressed by semidefinite constraints.
\begin{example}
Consider that $S = [0,1]$, $\mc{B} = [0,1]^3$, and
\[
\mc{M} = \{\mu\in \mc{P}([0,1]): \mathbb{E}_{\mu}[[\xi]_2]\in[0,1]^3\}.
\]
In this case, $\mathscr{R}_2([0,1]) = \mathscr{S}[\xi-\xi^2]_1$
and $\overline{cone(\mc{B})} = cone(\mc{B}) = \re_+^3$. Hence
\[
K = \left\{ y= \bbm y_0\\y_1\\y_2\ebm\in\re_+^3:
M_1[y] = \bbm y_0 & y_1\\y_1 & y_2 \ebm\succeq 0,\, 
L_{\xi-\xi^2}^{(1)}[y] = y_1-y_2\ge 0
\right\}.
\]
Since $S,\mc{B}$ are compact and $K$ has a nonempty interior, 
we must have 
\[ K^* = \mathscr{P}_2(S)+\mc{B}^*. \]
This can also be computed by the inner product of the Euclidean space:
\[\begin{aligned}
	K^* &= \re_+^3+\left\{\bbm q_0\\q_1\\q_2\ebm\in\re^3:
	\bbm q_0 & \frac{1}{2}q_1\\ \frac{1}{2}q_1 & q_2\ebm\succeq 0
	\right\}+\left\{\bbm 0\\q_1\\q_2\ebm\in\re^3:
	q_1 = -q_2\ge 0
	\right\}\\
	&= cone(\mc{B})^*+ \{q = q_0+q_1\xi+q_2\xi^2: q\in \Sigma[\xi]_2\} +
	\{q_1(\xi-\xi^2): q_1\ge 0\}\\
	& = \mc{B}^*+\mathscr{P}_2(S).
\end{aligned}\]
However, it is much more efficient to express $K^*$ as a nonnegative polynomial cone
under the bilinear operation as in \eqref{eq:<f,y>}.
\end{example}

For convenience, we make the following assumption.
\begin{assumption}\label{ass:closed}
Suppose $\cone{\mc{B}}$ is closed, $K^* = \mathscr{P}_d(S)+\mc{B}^*$, 
and $S$ is a compact semi-algebraic set defined by 
\begin{equation}\label{eq:S}
	S = \{\xi\in\re^n: g_1(\xi)\ge 0,\ldots, g_{m_1}(\xi)\ge 0\}
\end{equation}
for a polynomial tuple $g = (g_1,\ldots, g_{m_1})^{\top}$.
\end{assumption}
Under the above assumption, the nonnegative polynomial cone 
$\mathscr{P}_d(S)$ can be approximated 
by the quadratic module of $g$.
Indeed, if $\qmod{g}$ is archimedean, 
then by \cite[Proposition~8.2.1]{Niebook}, we have
\[
\mbox{int}\big(\mathscr{P}_d(S)\big) 
\subseteq \bigcap\limits_{k\ge \lceil d/2\rceil}
\big(\qmod{g}_{k}\cap \re[\xi]_d\big)
\subseteq \mathscr{P}_d(S),
\]
where $\mbox{int}(\cdot)$ denotes the interior of the set.
For $k\ge \max\{\lceil d/2\rceil, \lceil \deg(g)/2\rceil\}$, 
we define the $k$-th order SOS approximation of \eqref{eq:DRO_equiv} 
as 
\begin{equation}\label{eq:DRO_SOS_k}
	\left\{\begin{array}{cl}
		\min\limits_{(t,x,w)} & t\\
		\st & t\in \re,\, x\in\re^n,\, 
		w^{\top}[\xi]_d \in \qmod{g}_{k} + \mc{B}^*,\\
		& x\in X = \{x\in\re^n: e^{\top}x = 1,\,x\ge 0\},\\
		& (a_je_1\cdot t+A_jx+c_j-w)^{\top}[\xi]_d\in \qmod{g}_{k},\,\,
		\forall j\in [m].
	\end{array}
	\right.
\end{equation}
In the above, the constraint $w^{\top}[\xi]_d\in \qmod{g}_{k}+\mc{B}^*$
describes the membership of the polynomial $w^{\top}[\xi]_d$ in $\xi$,
belongs to the cone $\qmod{g}_{k}+\mc{B}^*\subseteq \re[\xi]_{k}$.
Notice that $\mathscr{S}[g]_{k}$ defined in \eqref{mom:S(g):2d}
is the dual cone of $\qmod{g}_k$.
Thus, the dual problem of \eqref{eq:DRO_SOS_k} is 
\begin{equation}\label{eq:DRO_mom_k}		
	\left\{\begin{array}{cl}
		\max\limits_{\substack{(\gamma_0,\gamma,y_1,\ldots, y_m,\\
				z_1,\ldots, z_m)}}& 
		\gamma_0-\sum\limits_{j=1}^m c_j^{\top}y_j \\
		\st & \gamma_0\in\re,\, \gamma\in\re^n,\, \gamma\ge 0,\\
		& \sum\limits_{j=1}^m a_j(e_1^{\top}y_j) = 1,\,
		\gamma_0e+\gamma+\sum\limits_{j=1}^m A_j^{\top}y_j = 0,\\
		& y_j = z_j|_d,\, z_j\in \mathscr{S}[g]_{k},\, \forall j\in[m],\\
		& \sum\limits_{j=1}^m y_j\in \operatorname{cone}(\mc{B}),\,
		\sum\limits_{j=1}^m z_j\in \mathscr{S}[g]_k,
	\end{array}\right.
\end{equation}	
where $z|_d \coloneqq (z_{\alpha})_{|\alpha|\le d}$ denotes 
the truncation of $z$ up to degree $d$.
The problem \eqref{eq:DRO_mom_k} is called the $k$-th order 
moment relaxation of \eqref{eq:DRO_dual}.
\begin{theorem}	\label{key1}
	Under Assumption~\ref{ass:closed},
	suppose $(\gamma_0^{(k)},\gamma^{(k)},y_1^{(k)},\cdots, y_m^{(k)},z_1^{(k)},\ldots, z_m^{(k)})$
	is an optimizer of \eqref{eq:DRO_mom_k}.
	Then $(\gamma_0^{(k)}, \gamma^{(k)},y_1^{(k)},\ldots, y_m^{(k)})$ 
	is a maximizer of \eqref{eq:DRO_dual}
	if and only if each $y_j^{(k)}$ belongs to $\mathscr{R}_d(S)$.
\end{theorem}	
\begin{proof}	
	Under Assumption~\ref{ass:closed}, $\mathscr{R}_d(S)$ is closed.
	If  $(\gamma_0^{(k)}, \gamma^{(k)},y_1^{(k)},\ldots, y_m^{(k)})$ 
	is a maximizer of \eqref{eq:DRO_dual}, 
	then each $y_j^{(k)} \in \overline{K}\subseteq \mathscr{R}_d(S)$. 
	Conversely, if each $y_j^{(k)}$ belongs to $\mathscr{R}_d(S)$,  
	then $\sum_{j=1}^m y_j^{(k)}\in  \mathscr{R}_d(S)$ since 
	$\mathscr{R}_d(S)$ is a convex cone.
	This implies $(\gamma_0^{(k)}, \gamma^{(k)},y_1^{(k)},\ldots, y_m^{(k)})$ 
	is feasible for \eqref{eq:DRO_dual}.
	So the relaxation \eqref{eq:DRO_mom_k} is tight, $(\gamma_0^{(k)}, \gamma^{(k)},y_1^{(k)},\ldots, y_m^{(k)})$ 
	is a maximizer of \eqref{eq:DRO_dual}.
\end{proof}

\section{A semidefinite algorithm}
\label{sc:momsos}

In this section, we propose a semidefinite algorithm to solve 
the distributionally robust shortfall risk problem \eqref{eq:DRO shortfall risk}
.
Suppose Assumption~\ref{ass:closed} holds. Denote
\[
t_0\coloneqq \lceil d/2\rceil,\quad t_c \coloneqq \lceil \deg(g)/2\rceil.
\]
\begin{alg}\label{alg}
Let $k \coloneqq \max\{t_0,t_c\}$,
$l_j \coloneqq k+1\,(\forall j\in [m])$,
and choose a generic polynomial $R\in\Sigma[\xi]_{2t_0+2}$.
Then do the following.
\begin{description}
	
	\item[Step~1]
	Solve the dual pair \eqref{eq:DRO_SOS_k}--\eqref{eq:DRO_mom_k} 
	at the relaxation order $k$
	to get minimizers $(t^{(k)},x^{(k)},w^{(k)})$ and $(\gamma_0^{(k)},\gamma^{(k)},y_1^{(k)}, \ldots, y_m^{(k)}, z_1^{(k)},\ldots, z_m^{(k)})$ respectively.
	
	\item[Step~2]
	For each $j\in[m]$, if $y_j^{(k)}\not=0$, solve the moment optimization problem
	\begin{equation}\label{eq:ATMP}
		\left\{\begin{array}{cl}
			\min\limits_{v_j} & \langle R, v_j\rangle \\
			\st& v_j|_{d} = y^{(k)}_j, \,
			v_j \in \mathscr{S}[g]_{l_j}.
		\end{array}\right.
	\end{equation}
	If the problem \eqref{eq:ATMP} is infeasible for some $j\in [m]$,
	then update $k \coloneqq k+1$ and go back to Step 1. 
	Otherwise, solve for a minimizer $v_j^{(k)}$.
	Then go to the next step.
	
	\item[Step~3]
	For each computed $v_j^{(k)}$, check whether there is an integer 
	$s_j \in\left[\max \{t_0,t_c\}, l_j\right]$ that makes
	\begin{equation}\label{eq:flat_trun}
		\operatorname{rank} M_{s_j-t_c}\big[v_j^{(k)}\big]
		= \operatorname{rank} M_{s_j}\big[v_j^{(k)}\big].
	\end{equation}
	If any $s_j$ does not exist, update $l_j \coloneqq l_j+1$, 
	return Step~2 for the specific $j$. 
	
	\item[Step~4]
	Let $y^*\coloneqq \sum_{j=1}^m y_j^{(k)}$. 
	Find the unique finitely atomic measure $\mu^*$ supported on $S$
	that is admitted by $y^*/(y_0^*)$.
	Return $(t^*,x^*,w^*) \coloneqq (t^{(k)}, x^{(k)}, w^{(k)})$,
	$\gamma_0^*\coloneqq \gamma^{(k)}$, $\gamma^*\coloneqq \gamma^{(k)}$
	and $y_j^* \coloneqq y_j^{(k)}$ for each $j\in[m]$.
\end{description}
\end{alg}
We would like to make some remarks for the algorithm.

In Step~1, we may assume the semidefinite programs 
\eqref{eq:DRO_SOS_k}--\eqref{eq:DRO_mom_k} are solvable without loss of generality, 
up to a proper selection of $\mc{B}$.

The Step~2 and Step~3 are used to verify if the computed optimizers
$y_j^{(k)}$ belongs to $\mathscr{R}_d(S)$ or not.
If $y_j^{(k)} = 0$, then it admits the zero measure.
Otherwise, we solve \eqref{eq:ATMP} to find a tms extension $v_j^{(k)}$ of $y_j^{(k)}$.
The rank condition \eqref{eq:flat_trun} is called {\it flat truncation} condition
(see \cite{Nieflat13}), 
where $M_{s_j-t_c}[v_j^{(k)}]$ is the moment matrix introduced in 
\eqref{eq:mommat}.
This is a sufficient condition for $v_j^{(k)}|_{2s_j}\in\mathscr{R}_{2s_j}(S)$. 
Since $2s_j\ge d$ and $y_j^{(k)} = v_j^{(k)}|_d$, 
when the flat truncation condition \eqref{eq:flat_trun} holds, 
we must have $y_j^{(k)}$ admits a measure supported on $S$.
For the special case that $m=1$, the loss function $l(Z)$ becomes
a single polynomial. 
If $y_1^{(k)}\in \mathscr{R}_d(S)$, then it must correspond to a worst-case
measure. Such kinds of polynomial distributionally robust optimization problems
have been discussed in \cite{NieYang23,NieZhong23}.
However, if $m>1$, then we need to verify if $y_j^{(k)}\in \mathscr{R}_d(S)$
for each $j\in[m]$ to ensure \eqref{eq:DRO_mom_k} is a tight relaxation for the original 
DRO shortfall risk problem.

If $y_j^{(k)}\in \mathscr{R}_d(S)$ for all $j\in[m]$,
then $y^* = \sum_{j=1}^m y_j^{(k)}$ also belongs to $\mathscr{R}_d(S)$.
Thus, the tms $y^*/(y_0^*)$ admits a unique finitely atomic measure 
$\mu^*$ supported on $S$ \cite[Theorem~2.7.1]{Niebook}.
In other words, there are $r$ distinct points $u_1,\ldots, u_r\in S$ 
and positive scalars $\theta_1,\ldots, \theta_r\in\re$
such that 
\[
\mu^* = \sum\limits_{i=1}^m \theta_i\delta_{u_i}
\quad\mbox{and}\quad
y_{\alpha}^* = y_0^* \int_S \xi^{\alpha}{\tt d}\mu^*\,
(\forall \alpha\in \N_d^p),
\]
where $\delta_{u_i}$ denotes the Dirac measure supported at $u_i$.
By Theorem~\ref{thm:equiv}, $\mu^*$ is a worst-case measure, 
which can be computed via Schur decomposition. 
The computation can be implemented using the {\tt MATLAB} 
software package {\tt GloptiPoly3} \cite{HenrionLasserre2009}. 
For further details on this approach, we refer the reader to \cite{NieAtrunc}.

\begin{theorem}\label{thm:finite}
	Under Assumption~\ref{ass:closed}, suppose Algorithm~\ref{alg} 
	terminates in finite loops with optimizers 
	$(t^*, x^*,w^*)$ and $(\gamma_0^*, \gamma^*, y_1^*,\ldots, y_j^*)$.
	If $t^* = \gamma_0^*-\sum_{j=1}^m c_j^{\top}y_j^*$,
	then $(t^*, x^*)$ is an optimizer of the original 
	DRO problem \eqref{eq:SC_DRO shortfall risk}.
\end{theorem}	
\begin{proof}		
	Let $f_1, f_2$ be optimal values of the problems \eqref{eq:DRO_equiv} 
	and \eqref{eq:DRO_dual} respectively. 
	Then $f_1\ge f_2$ by the weak duality.
	The termination criteria of Algorithm~\ref{alg} ensures 
	$y_j^*\in \mathscr{R}_d(S)$ for each $j\in[m]$.
	By Theorem~\ref{key1}, $(\gamma_0^*, \gamma^*, y_1^*,\ldots, y_m^*)$ 
	is a maximizer of \eqref{eq:DRO_dual}. 
	Hence we have $f_2 = \gamma_0^*-\sum_{j=1}^m c_j^{\top}y_j^*$.
	Since $\qmod{g}_{k}\cap \re[\xi]_d\subseteq \mathscr{P}_d(S)$, 
	the optimizer $(t^*,x^*, w^*)$ computed by Algorithm~\ref{alg}
	is also feasible for \eqref{eq:DRO_equiv}. 
	If $t^* = \gamma_0^*-\sum_{j=1}^m c_j^{\top}y_j^*$, 
	then $t^*\ge f_1 \geq f_2 = t^*$, thus $t^* = f_1$.
	This implies that $\left(t^*, x^*, w^*\right)$ is a minimizer of \eqref{eq:DRO_equiv}. 
	By Theorem~\ref{thm:equiv}, $(t^*,x^*)$ is also an optimizer 
	of \eqref{eq:SC_DRO shortfall risk}.
\end{proof}

\begin{theorem}\label{theo:C2}
	Under Assumption~\ref{ass:closed}, 
	suppose $\qmod{g}$ is archimedean and there exists a feasible point 
	$(\hat{t}, \hat{x}, \hat{w})$ of (\ref{eq:DRO_SOS_k}) such that 
	\begin{equation}\label{eq:q1q2}
		\hat{w}^{\top}[\xi]_d = q_1(\xi)+q_2(\xi)
		\quad\mbox{with}\quad
		q_1(\xi)>0\,\,(\forall \xi\in S),\,\,q_2(\xi) \in \mc{B}^*.
	\end{equation}
	Suppose $(\gamma_0^{(k)}, \gamma^{(k)},y_1^{(k)}, \ldots, y_m^{(k)}, z_1^{(k)},\ldots, z_m^{(k)})$ 
	is an optimizer of the $k$-th order moment relaxation \eqref{eq:DRO_mom_k}. 
	Then $\{y_j^{(k)}\}_{k=1}^{\infty}$ is bounded for each $j$ and 
	every accumulation point of 
	$\{(\gamma_0^{(k)}, \gamma^{(k)},y_1^{(k)}, \ldots, y_m^{(k)})\}_{k=1}^{\infty}$ 
	is a maximizer of \eqref{eq:DRO_dual}.
\end{theorem}
\begin{proof}
	Since $(\hat{t}, \hat{x}, \hat{w})$ is feasible for (\ref{eq:DRO_SOS_k}),
	we have
	\[
	(a_je_1\cdot \hat{t}+A_j\hat{x}+c_j-\hat{w})^{\top}[\xi]_d \in \operatorname{QM}[g]_{k},
	\quad \forall j\in[m].
	\]
	For every  $(\gamma_0,\gamma, y_1,\ldots, y_m, z_1,\ldots, z_m)$
	that is feasible to \eqref{eq:DRO_mom_k}, we have
	\begin{align}
		\label{eq:nnp_tms_nonn}
		(a_je_1\cdot \hat{t}+A_j\hat{x}+c_j-\hat{w})^{\top}y_j\,\ge\, 0,\quad \forall j\in[m];\\
		\label{eq:reform_diff} 
		\hat{t} - \gamma_0+\sum\limits_{j=1}^m c_j^{\top}y_j\,\ge\, 
		\sum\limits_{j=1}^m (a_je_1\cdot \hat{t}+A_j\hat{x}+c_j)^{\top}y_j. 
	\end{align}
	The inequality \eqref{eq:nnp_tms_nonn} follows the dual relation \eqref{eq:qmod_dual},
	and the second inequality \eqref{eq:reform_diff} is obtained by introducing an intermediate term 
	$\sum_{j=1}^m (a_je_1\cdot t+A_jx)^{\top}y_j$. 
	For convenience, denote $\bar{y} = \sum_{j=1}^m y_j$. 
	The decomposition \eqref{eq:q1q2} implies
	\[
	\sum\limits_{j=1}^m (a_je_1\cdot \hat{t}+A_j\hat{x}+c_j)^{\top}y_j
	\,\ge\, \sum\limits_{j=1}^m \hat{w}^{\top}y_j 
	\,=\, \langle q_1, \bar{y}\rangle
	+ \langle q_2, \bar{y}\rangle.
	\]
	Notice that $q_2\in \mc{B}^*$ by assumption.
	Since $\bar{y}\in\operatorname{cone}(\mc{B})$, 
	we have $\langle q_2,\bar{y}\rangle \ge 0$.
	Since $\qmod{g}$ is archimedean and the polynomial $q_1(\xi)$ is positive over $S$,
	for any small $\epsilon>0$, there is a positive integer $k_0$ such that 
	$q_1(\xi)-\epsilon\in \qmod{g}_{2k_0}$.
	Let $\bar{y}^{(k)} = \sum_{j=1}^m y_j^{(k)}$.
	For all $k\ge k_0$, it holds that
	\[
	\langle q_1, \bar{y}^{(k)} \rangle
	= \langle q_1-\epsilon\cdot 1, \bar{y}^{(k)}\rangle + 
	\epsilon \langle 1, \bar{y}^{(k)} \rangle 
	\geq \epsilon \bar{y}_0^{(k)}
	= \epsilon \sum_{j=1}^m y_{j,0}^{(k)}.
	\]
	In the above, $1$ stands for the scalar polynomial $q=1$ and
	$\bar{y}^{(k)}_0, y_{j,0}^{(k)}$ denote the first entry of $\bar{y}^{(k)}$ 
	and $y_j^{(k)}$respectively.
	Let $f_2$ be the optimal value of \eqref{eq:DRO_dual}.
	Since \eqref{eq:DRO_mom_k} is a relaxation of \eqref{eq:DRO_dual},
	by \eqref{eq:reform_diff} and the previous analysis, we get
	\[
	\hat{t}-f_2 \,\ge\, \hat{t}-\gamma_0^{(k)}+\sum_{j=1}^m c_j^{\top}y_j^{(k)}
	\,\ge\, \langle q_1, \bar{y}^{(k)}\rangle
	\,\ge\, \epsilon \sum\limits_{j=1}^m y_{j,0}^{(k)}.
	\]
	The feasibility of \eqref{eq:DRO_mom_k} ensures that $y_{j,0}^{(k)}\ge 0$.
	Thus, each $\{y_{j,0}^{(k)}\}_{k=1}^{\infty}$ is bounded, i.e.,
	\[
	0\le y_{j,0}^{(k)}\leq \frac{1}{\epsilon}\left(\hat{t}-f_2\right),
	\quad j=1,\ldots, m.
	\]
	Then, by \cite[Theorem~4.7]{NieYang23}, the sequence $\{y_j^{(k)}\}_{k=1}^{\infty}$ is bounded for each $j$ and 
	every accumulation point of 
	$\{(\gamma_0^{(k)}, \gamma^{(k)},y_1^{(k)}, \ldots, y_m^{(k)})\}_{k=1}^{\infty}$ 
	is a maximizer of \eqref{eq:DRO_dual}.
\end{proof}

\section{Numerical Experiments}
\label{sc:numerical}
In this section, we will use the Algorithm~\ref{alg} to run numerical experiments 
on distributionally robust shortfall risk portfolio problems.
The experiments were conducted on a laptop equipped with an Intel(R) Core(TM) i5-6500 CPU 
and 8.00 GB of RAM (7.87 GB available), using \texttt{MATLAB} with the 
\texttt{GloptiPoly3} \cite{HenrionLasserre2009} and \texttt{SeDuMi} \cite{Sturm1999} packages.
We construct the moment feasible set $\mc{B}$ from statistical
samples using the following method.
Suppose 
\[
T = \left\{\xi^{(1)}, \ldots, \xi^{(N)}\right\}
\] 
is a given sample set of $\xi$. 
We randomly generate subsets $T_1, \ldots, T_s \subseteq T$ such that 
each $T_i$ contains $\lceil N / 2\rceil$ samples. 
For a chosen degree $d$, build the vectors of moment bounds
$\bar{l} = (\bar{l}_{\alpha}),\bar{u} = (\bar{u}_{\alpha})$ 
for $\alpha\in\N_d^p$ with each
\begin{equation}\label{eq:lu}
\begin{aligned}
	\bar{l}_\alpha = \min\limits_{j\in[s]}
	\Bigg\{ \frac{1}{\left|T_j\right|} \sum_{\xi^{(i)} \in T_j}(\xi^{(i)})^\alpha, 
	\frac{1}{\left|T \backslash T_j\right|} \sum_{\xi^{(i)} \in T \backslash T_j} 
	(\xi^{(i)})^\alpha \Bigg\}, \\
	\bar{u}_\alpha = \max\limits_{j\in [s]}
	\Bigg\{
	\frac{1}{\left|T_j\right|} \sum_{\xi^{(i)} \in T_j} (\xi^{(i)})^\alpha, 
	\frac{1}{\left|T \backslash T_j\right|} \sum_{\xi^{(i)} \in T \backslash T_j} 
	(\xi^{(i)})^\alpha \Bigg\}.
\end{aligned}
\end{equation}
This implies the moment feasible set
\begin{equation}\label{Y}
\mc{B} = \left\{y \in \mathbb{R}^{\mathbb{N}_d^p}: 
y_0 = 1,\, \bar{l}_{\alpha} \leq y_{\alpha} \leq \bar{u}_{\alpha}\,
(\forall 0\not=\alpha\in \N_d^p) \right\}.
\end{equation}
It is easy to verify that 
\[
\cone{\mc{B}} = \{y\in \re^{\N_d^p}: y_0\ge 0,\, 
\bar{l}_{\alpha}\cdot y_0\le y_{\alpha}\le \bar{u}_{\alpha} \cdot y_0\,
(\forall 0\neq \alpha\in \N_d^p)\}.
\]
Assume that each \( \xi^{(i)}\) independently follows the distribution of \(\xi\).
By the convergence properties of sample average approximation, 
as the sample size $N$ increases,
the moment ambiguity set \( \mathcal{M} \) with \( \mc{B} \) constructed in \eqref{Y}
gives a better approximations of the true distribution of $\xi$. 

We first present a synthetic example.
\begin{example}\label{ex:artificial}
Consider three stocks labeled as A, B, C in the stock market
with returns denoted by $r_1,r_2,r_3$ respectively.
Suppose $r = (r_1,r_2,r_3)^{\top}$ depends on a random vector $\xi=(\xi_1,\xi_2)^{\top}$
via polynomial functions, where $\xi_1,\xi_2$ are independently and normally
distributed as 
\[
\xi_1,\xi_2 \sim \mathcal{N}_{[0,1]}(0.5,0.2).
\]
In the above, $\mathcal{N}_{[a,b]}(c, d)$ stands for the truncated 
normal distribution restricted to $[a,b]$ with 
the mean value $c$ and the standard deviation $d$.
Consider that
\begin{equation}
	\label{eq:return}
	\begin{aligned}
		&	r_{1}(\xi)=0.6\xi_{1}-\xi_{2}+2\xi_{1}^{2}+\xi_{1}\xi_{2}-\xi_{2}^{2}+0.5 ,~\\
		&	r_{2}(\xi)=0.2\xi_{1}+3\xi_{2}-2\xi_{1}^{2}-2\xi_{1}\xi_{2}+\xi_{2}^{2}-0.02,~\\ 
		&	r_{3}(\xi)=0.3\xi_{1}-1.5\xi_{2}+\xi_{1}^{2}+\xi_{1}\xi_{2}+2\xi_{2}^{2}+0.3.
	\end{aligned}
\end{equation}
We choose the confidence level and the loss function
\[
\lambda = 0.6,\quad l(Z) = \sup\{0.1Z+1, Z+0.1\}.
\]
Let $x=(x_1, x_2, x_3)^{\top}$ denote the investment proportion of these three stocks. 
The distributionally robust shortfall risk portfolio problem 
\eqref{eq:SC_DRO shortfall risk} becomes	
\begin{equation}\label{eq:example11}
	\left\{\begin{array}{rl}
		\min \limits_{(t, x) } & t \\
		\st & \mathbb{E}_{\mu_1}[0.1(x^{\top }r(\xi)+t)-0.4] + 
		\mathbb{E}_{\mu_2}[x^{\top }r(\xi)+t+0.5] \geq 0 \\
		& \forall \mu_1,\mu_2\in \mc{M}([0,1]^2),\,\mu_1+\mu_2\in\mc{M},\\
		& x_1+x_2+x_3=1,~ x_1,x_2,x_3\geq0,~t\in \mathbb{R}.
	\end{array}\right.
\end{equation}
We testify the performance of this model and 
evaluate the efficiency of our algorithm across varying sample sizes.
Let $\mc{M}$ be a moment ambiguity set constructed with the moment 
feasible set $\mc{B}$ as in \eqref{Y}.
We consider the following cases: 
\[
s = 3,\quad d=n=2,\quad N\in \{100, 600, 1000\}.
\]
In computations, we use the \texttt{MATLAB} functions \texttt{makedist} and \texttt{truncate} 
to randomly generate $1000$ samples of $\xi$ and form a shared sample set for all 
three cases. For each case, the truncated moment bounds $\bar{l},\bar{u}$ as in \eqref{eq:lu}.

(i) When $N=100$, we obtain the moment bounds
\[\begin{aligned}
	&	\bar{l}=  (   0.4615,
	0.4883,
	0.2549,
	0.2299,
	0.2796)^{\top}, \\
	&	\bar{u}=  ( 0.5147,
	0.5384,
	0.3067,
	0.3034,
	0.3358)^{\top} .
\end{aligned}
\]
We ran Algorithm~\ref{alg} to solve \eqref{eq:example11}.
It terminated at the initial relaxation order $k=1$ with the computed optimal value and 
optimal solutions
\[
\begin{aligned}
	&	t^* \approx   3.2176,\quad x^*\approx (0.0000, 0.3333, 0.6667)^{\top},\\
	&		y_1^{*} \approx(10.0000,
	4.6150,
	4.8908,
	2.5544,
	2.6292,
	2.7960)^{\top},\\
	&	y_2^{*} \approx(0.0000,
	0.0000,
	0.0000,
	0.0000,
	0.0000,
	0.0000)^{\top}.
\end{aligned}
\]
The nonzero tms $y^{*} = y_1^*+y_2^*$, 
corresponds to the worst-case measure 
\[
\mu^* =  0.1442\delta_{u_1}+0.0220\delta_{u_2}+0.8338\delta_{u_3},
\]
where $u_1 = (0.0000, 0.0000)^{\top}$, $u_2 = ( 0.0000, 0.6385)^{\top}$
and $u_3 = (  0.5535,0.5697)^{\top}$.

(ii) When $N=600$, we obtain the moment bounds
\[\begin{aligned}
	&		\bar{l}=  (  0.4852,
	0.4894,
	0.2744,
	0.2463,
	0.2765)^{\top}, \\
	&	\bar{u}=  ( 0.5107,
	0.5193,
	0.2999,
	0.2638,
	0.3089)^{\top}.
\end{aligned}
\]
We ran Algorithm~\ref{alg} to solve \eqref{eq:example11}.
It terminated at the initial relaxation order $k=1$ with the computed optimal value and 
optimal solutions
\[
\begin{aligned}
	&	t^* \approx  3.1959,\quad x^*\approx (0.5769, 0.2692, 0.1538)^{\top},\\
	&	y_1^{*} \approx( 10.0000,
	4.8520,
	4.9322,
	2.7440,
	2.4630,
	2.8886
	)^{\top}\\
	&	y_2^{*} \approx(0.0000,
	0.0000,
	0.0000,
	0.0000,
	0.0000,
	0.0000)^{\top}.
\end{aligned}
\]
The nonzero tms $y^{*} = y_1^*+y_2^*$
corresponds to the worst-case measure 
\[
\mu =  0.1118\delta_{u_1}+0.7328\delta_{u_2}+ 0.1554\delta_{u_3},
\]
where $u_1 = ( 0.0950, 0.0000)^{\top}$, $u_2 = (0.6031,0.4955)^{\top}$
and $u_3 = ( 0.2099,0.8373)^{\top}$.

(iii) When $N=1000$, we obtain the moment bounds
\[\begin{aligned}
	&	\bar{l}= (  0.4852,
	0.4894,
	0.2744,
	0.2463,
	0.2765)^{\top}, \\
	& \bar{u}= ( 0.5107,
	0.5193,
	0.2999,
	0.2638,
	0.3089)^{\top} . 
\end{aligned}
\]
We ran Algorithm~\ref{alg} to solve \eqref{eq:example11}.
It terminated at the initial relaxation order $k=1$ with the computed optimal value and 
optimal solutions
\[
\begin{aligned}
	&	t^* \approx  3.1946,\quad x^*\approx (0.5556, 0.3333, 0.1111)^{\top},\\
	&	y_1^{*} \approx( 10.0000,
	4.8360,
	4.9640,
	2.7630,
	2.4889,
	2.8986)^{\top},\\
	&	y_2^{*} \approx(
	0.0000,
	0.0000,
	0.0000,
	0.0000,
	0.0000,
	0.0000)^{\top}.
\end{aligned}
\]
The nonzero tms $y^{*} = y_1^*+y_2^*$
corresponds to the worst-case measure 
\[
\mu = 0.1396\delta_{u_1}+0.2571\delta_{u_2}+0.6033\delta_{u_3},
\] 
where $u_1 = (0.2412, 0.0000)^{\top}$, $u_2 = (0.2233, 0.6742)^{\top}$
and $u_3 = (0.6506,0.5355)^{\top}$.

From these computational results, it can be easily observed that the optimal
value $t^*$ decreases as the sample size increases.
This indicates that the approximation accuracy of ambiguity set improves 
as the sampling size increases. 
To further illustrate the impact of sample sizes on computational performance, 
we conducted 30 independent experimental trials for each distinct sample size 
configuration.
The computational results of optimal values are plotted in Figure~\ref{FF}.
\begin{figure}[htb]
	\centering
	\includegraphics[width=8.5cm]{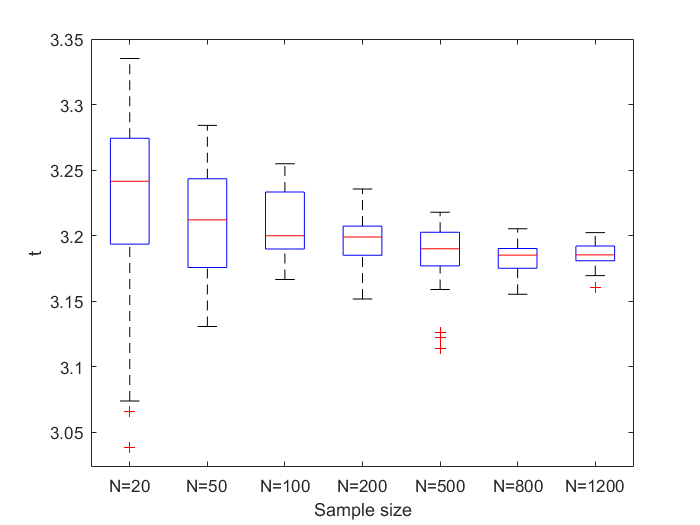}
	\caption{Trend of optimal values with
		increasing sample sizes of Example~\ref{ex:artificial}.}
	\label{FF}
\end{figure}
In the figure, the {\it box} represents the {\it interquartile range} 
(IQR, 25th to 75th percentiles), containing the middle $50\%$ of the data. 
It gradually narrows as the sample size increases, 
which indicates reduced data variability. 
The {\it horizontal line} within each box denotes the {\it median} (50th percentile), 
which progressively approaches a stable value with larger sample sizes. 
The {\it plus (+)} symbols mark {\it outliers} beyond 1.5×IQR, 
while the {\it whiskers} extend to the minimum and maximum non-outlier values. 
Collectively, the figure reveals a trend of stabilizing optimal values 
with increasing sample size.	
\end{example}

We then select five real-world stocks and choose four commonly used
influence factors as random variables. 
To describe the nonlinear relation between stock returns and these influencing factors,
we use polynomial fitting techniques to derive functions of returns based on historic data. 
The out-of-sample performance of the model is checked and analyzed subsequently.

\begin{example}\label{ex:real_stock1}
Consider five S$ \&$P 500 stocks from the American stock market: 
AAPL, MSFT, AMZN, C, JPM.
We collected their weekly historical returns from the AlphaVantage (\url{www.alphavantage.co})
from February 9, 2001 to October 26, 2007.
Consider four factors known to influence stock returns: 
Mkt-RF (Market Risk Factor), SMB (Small Minus Big), HML (High Minus Low) and
RMW (Robust Minus Weak). 
We collected historical data for the corresponding time period on the 
Kenneth French’s database (French, 2021).
Let $r=(r_1, r_2, r_3, r_4, r_5)^{\top}$ denote the 
vector of returns for these five stocks, 
and $\xi = (\xi_1, \xi_2, \xi_3, \xi_4)^{\top}$ the 
vectors of chosen influence factors.
By using the polynomial fitting based on the ordinary least squares (OLS) method, 
we build the following polynomial dependence of $r$ on $\xi$:

\[\begin{aligned}
	&\begin{array}{l}
		r_{1}(\xi) \approx 0.0141    +1.0008\xi_{1}   -0.0380\xi_{2}   -0.5000\xi_{3}   -0.2337\xi_{4}   -4.9797\xi_{1}^{2}\\
		\qquad +5.5816\xi_{1}\xi_{2}    +3.3308\xi_{1}\xi_{3}  -17.1760\xi_{1}\xi_{4}  -10.3220\xi_{2}^{2}   -0.4373\xi_{2}\xi_{3} \\
		\qquad-20.8000\xi_{2}\xi_{4}   +27.9150\xi_{3}^{2}  -23.6060\xi_{3}\xi_{4}  -31.1210\xi_{4}^{2},
	\end{array}\\				
	&\begin{array}{l}
		r_{2}(\xi) \approx    0.0025    +0.8397\xi_{1}   -0.2303\xi_{2}   -0.8592\xi_{3}   -0.3649\xi_{4}   -0.8987\xi_{1}^{2} \\
		\qquad -6.1472\xi_{1}\xi_{2}   -3.2933\xi_{1}\xi_{3}  -2.6637\xi_{1}\xi_{4}    +3.3105\xi_{2}^{2}  -13.1730\xi_{2}\xi_{3}  \\
		\qquad -1.7810\xi_{2}\xi_{4}   +33.2040\xi_{3}^{2}   -0.7877\xi_{3}\xi_{4}  -20.7270\xi_{4}^{2},
	\end{array}\\	
	&\begin{array}{l}
		r_{3}(\xi) \approx    0.0049    +0.8515\xi_{1}   -0.2664\xi_{2}   -1.2320\xi_{3}   -2.4132\xi_{4}   -4.3281\xi_{1}^{2} \\
		\qquad -1.6018\xi_{1}\xi_{2}    +8.6433\xi_{1}\xi_{3}   -25.8650\xi_{1}\xi_{4}   +14.8680\xi_{2}^{2}   -2.1372\xi_{2}\xi_{3} \\
		\qquad +51.6200\xi_{2}\xi_{4}   -2.2458\xi_{3}^{2}   +72.4570\xi_{3}\xi_{4}   +37.1640\xi_{4}^{2},
	\end{array}\\				
	&\begin{array}{l}
		r_{4}(\xi) \approx   -0.0001    +1.2593\xi_{1}   -0.4968\xi_{2}    +0.1437\xi_{3}   -0.1268\xi_{4}    +2.1782\xi_{1}^{2} \\
		\qquad  -8.1962\xi_{1}\xi_{2}    +7.3321\xi_{1}\xi_{3}   +3.3344\xi_{1}\xi_{4}    +1.9698\xi_{2}^{2}    +2.2514\xi_{2}\xi_{3} \\
		\qquad -14.7650\xi_{2}\xi_{4}  -13.3260\xi_{3}^{2}   +42.9830\xi_{3}\xi_{4}    +1.4268\xi_{4}^{2},
	\end{array}\\				
	&\begin{array}{l}
		r_{5}(\xi) \approx   -0.0007    +1.5127\xi_{1}   -0.4844\xi_{2}    +0.7546\xi_{3}   -0.5325\xi_{4}    +0.3397\xi_{1}^{2}  \\
		\qquad  -1.1218\xi_{1}\xi_{2}   -5.0506\xi_{1}\xi_{3}   -15.4610\xi_{1}\xi_{4}   +10.6180\xi_{2}^{2}  -15.5300\xi_{2}\xi_{3} \\
		\qquad +22.4820\xi_{2}\xi_{4}  -19.0280\xi_{3}^{2}   +10.8080\xi_{3}\xi_{4}    +1.9758\xi_{4}^{2}.
	\end{array}\\	
\end{aligned}
\]		
Suppose the random vector is supported on 
\[
S = \{\xi\in\re^4 : \left(\xi-\mu_M\right)^{\top} \Sigma_M^{-1}\left(\xi-\mu_M\right) \leq \bar{R}^2 \}, 
\]	
where  $\bar{R}$ is the ellipsoidal set radius, 
$\mu_M$ is the sample mean vector estimate of $\xi$ 
and $\Sigma_M$ is the sample covariance matrix estimate of $\xi$. 
Based on historic data, we set $\bar{R} = 3$ and compute
\[
\mu_M=\bbm 0.0006\\
0.0013\\
0.0012\\
0.0012\ebm,\quad
\Sigma_M= 10^{-3}\cdot\left[\begin{array}{rrrr}
	0.4605   & 0.0301  & -0.0621 &  -0.1135\\
	0.0301   & 0.1407    &0.0125   &-0.0335\\
	-0.0621   & 0.0125   & 0.0994   & 0.0237\\
	-0.1135   &-0.0335   & 0.0237 &   0.1561			
\end{array}\right].
\]	
For $s = 3, n=2, d=4$, we construct the moment feasible set $\mc{B}$ as in \eqref{Y} with
\[
\begin{aligned}
	& \begin{array}{l}\bar{l}= 10^{-3}\cdot\ (
		-1.4987,
		-1.0721,
		-0.7648,
		-0.7738,
		0.2471,
		-0.0236,
		-0.0999,\\
		\qquad \qquad \quad\,\,\,-0.2136,
		0.1000,
		-0.0123,
		-0.0683,
		0.0449,
		-0.0020,
		0.0209
		)^{\top},\end{array} \\
	& 	\begin{array}{l}\bar{u}= 10^{-3}\cdot 
		(3.4969, 
		2.3595,
		1.8434,
		2.6445,
		0.6731,
		0.0936,
		-0.0163,
		0.0095,\\
		\qquad \qquad \quad\,\,\, 0.1840,
		0.0384,
		-0.0058,
		0.1564,
		0.0444,
		0.2742
		)^{\top}.\end{array}
\end{aligned}
\]
Let $x=(x_1,x_2,x_3,x_4,x_5)^{\top}$ represent the proportion of investment in five stocks.
We choose the confidence level and the loss function
\[
\lambda = 1,\quad l(Z) = \sup\{Z-0.5, 0.5Z+2\}.
\] 
The distributionally robust shortfall risk portfolio problem 
\eqref{eq:SC_DRO shortfall risk} becomes
\begin{equation}\label{eq:example22}
	\left\{\begin{array}{rl}
		\min \limits_{(t, x)} & t \\
		\st & \mathbb{E}_{\mu_1}[x^{\top }r(\xi)+t+1.5]+\mathbb{E}_{\mu_2}[0.5(x^{\top }r(\xi)+t)-1] \geq 0\\
		& \forall \mu_1,\mu_2\in \mc{M}(S),\,\mu_1+\mu_2\in \mc{M},\\
		& x_1+x_2+x_3+x_4+x_5=1,~\\
		& x_1,x_2,x_3,x_4,x_5\geq0,~t\in \mathbb{R}.~
	\end{array}\right.
\end{equation}
We ran Algorithm~\ref{alg} to solve \eqref{eq:example22}.
It terminated at the initial relaxation order $k=1$
with the computed optimal value and the optimal solution
\[
\begin{aligned}
	& \,\, \,t^{*} \approx 2.0007,\quad
	x^{*} \approx(
	0.5902,
	0.0000,
	0.4098,
	0.0000,
	0.0000)^{\top},\\
	&\begin{array}{l}
		y_1^{*}\approx(
		0.0000,
		0.0000,
		0.0000,
		0.0000,
		0.0000,
		0.0000,
		0.0000,
		0.0000,
		0.0000,\\
		\quad\quad \,\,\,	0.0000,
		0.0000,
		0.0000,
		0.0000,
		0.0000,
		0.0000)^{\top}.\end{array}\\
	&  
	\begin{array}{l}
		y_2^*\approx 10^{-2}\cdot(
		200.0000,
		-0.2997,
		0.4719,
		0.3687,
		0.5289,
		0.1346,
		-0.0047,\\
		\quad\quad \,\,\, -0.0200,
		0.0019,
		0.0354,
		0.0077,
		-0.0137,
		0.0090,
		-0.0004,
		0.0548)^{\top}.
	\end{array}
\end{aligned}
\]	
The nonzero tms $y^*=y_1^*+y_2^*$
corresponds to the worst-case measure 
\[
\mu^* = \delta_{u},\quad
\mbox{where}\quad u = (   -0.0015,
0.0024,
0.0018,
0.0026)^{\top}.
\]		
\end{example}

To show the efficiency of our model involving polynomial dependence between 
stock returns and their influence factors, 
we compare it with the sample average approximation (SAA) model and 
the classic distributionally robust shortfall risk model:
\begin{equation}\label{eq:DROclassic}
\left\{\begin{array}{cl}
	\min\limits_{t\in \re} & t\\
	\st & \sup\limits_{\nu\in \mc{M'}}\mathbb{E}_{\nu}[l(x^{\top}r-t)]\le \lambda,\\
	& x \in X,
\end{array}
\right.
\end{equation}
where the stock return $r\in\re^n$ itself is considered as a random vector
following the distribution of some measure $\nu\in \mc{M}'$.

\begin{example}\label{ex:real_stock2}
Consider the same stocks along with their historic data in Example~\ref{ex:real_stock1}.
We compare our model \eqref{eq:DRO shortfall risk} with the SAA model and \eqref{eq:DROclassic}
via their out-of-sample performance in form of the cumulative return. 
Mathematically, the cumulative return $R_c$ of a portfolio of \(n\) stocks over 
\(\mathcal{T}\) selected time periods is computed as
\[
R_c = -1+\prod_{t = 1}^{\mathcal{T}}\Big(1 + \sum_{i = 1}^{n}x_i\cdot r_{i,t}\Big),
\]
where \(x_i\) is the investment proportion of the \(i\)-th stock and \(r_{i,t}\) 
is the return of the \(i\)-th stock on the \(t\)-th time period.

We carried out an out-of-sample evaluation using the rolling window method 
\cite{GuoXu19} over a period of $250$ weeks. 
This involves using the first $250$ data instances to infer the optimal portfolio strategy 
for the $251{\mbox{st}}$ week. 
Precisely, we consider the following DRO model without influence factors: 
\begin{equation}\label{eq:DRO_stock_nofac}
	\left\{\begin{array}{rl}
		\min\limits_{(t,x)} & t \\
		\st & 
		\mathbb{E}_{\nu_1}[x^{\top }r+t+1.5] +
		\mathbb{E}_{\nu_2}[0.5(x^{\top }r+t)-1] \geq 0\\
		& \forall \nu_1,\nu_2\in\mc{M}(S),\, \nu_1+\nu_2\in \mc{M}',\\
		& x_1+x_2+x_3+x_4+x_5=1,~\\
		& x_1,x_2,x_3,x_4,x_5\geq0,~t\in \mathbb{R}.~
	\end{array}\right.
\end{equation}
In the above, the ambiguity set $\mc{M'}$ is similarly constructed with $\mc{B}$ in \eqref{Y}. 
The SAA model with influence factors is written as 
\begin{equation}\label{eq:SAA}
	\left\{\begin{array}{rl}
		\min \limits _{(t,x)}& t \\
		\st &  \frac{1}{N}\sum\limits_{i = 1}^{N}[(x^{\top }r(\xi^{(i)})+t+1.5)+(0.5(x^{\top }r(\xi^{(i)})+t)-1)] \geq 0,~\\
		& x_1+x_2+x_3+x_4+x_5=1,~\\
		& x_1,x_2,x_3,x_4,x_5\geq0,~t\in \mathbb{R}.~
	\end{array}\right.
\end{equation}
where $N$ is the sample size.
In the experiment, we use the same polynomial relationship as in Example~\ref{ex:real_stock1} 
for all the time windows, while updating ambiguity sets with the change of samples.

The computational results are presented in Figure~\ref{fig:D1}.
\begin{figure}[htb]
	\centering
	\includegraphics[width=8.5cm]{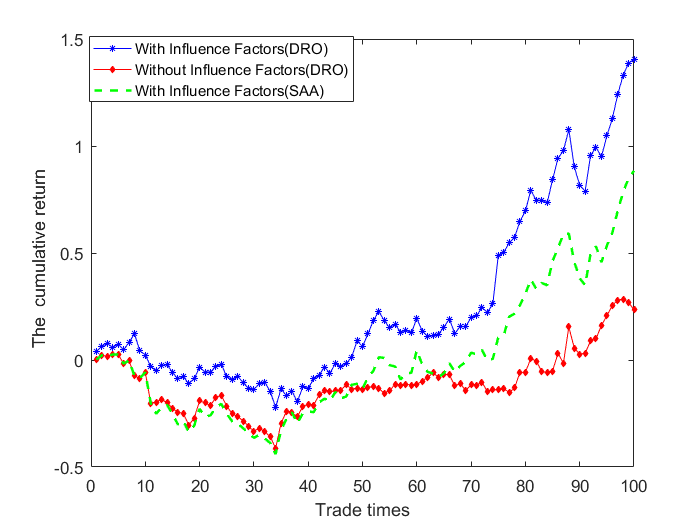}
	\caption{Comparison of out-of-sample performance for \eqref{eq:DRO shortfall risk}, \eqref{eq:SAA}
		and \eqref{eq:DROclassic} in Example~\ref{ex:real_stock2}.}
	\label{fig:D1}
\end{figure}
We use {\it asterisk} markers to plot the cumulative returns 
of the DRO model \eqref{eq:example22} with influence factors,
{\it dotted} line to plot the cumulative returns of the
SAA model \eqref{eq:SAA}, 
and use {\it diamond} markers to plot the cumulative returns of 
the classic DRO model \eqref{eq:DRO_stock_nofac}.
In Figure \ref{fig:D1}, the cumulative returns of these three models were 
close at the beginning. However, over time, the trajectory changed. 
While all three models suggest investment proportions 
maintain upward trends in cumulative returns,
the growth of other two baseline models was relatively moderate. 
In particular, during most of the observation periods, 
the cumulative rates of return for other two baseline models were consistently lower 
than that of our factor-inclusive DRO model. 
These results indicate that our model outperforms the long-term 
cumulative returns of the other two models.
\end{example}

\section{Conclusions}\label{sc:con}
In this paper, we consider a distributionally robust portfolio optimization 
problem based on shortfall risk and polynomial return functions.
Assume the ambiguity is given by moment information and the loss functions are
piecewise linear in the decision variables.
We show that this DRO model can be transformed into a dual pair of solvable linear 
conic optimization problems with moment and nonnegative polynomial cones.
We propose a semidefinite algorithm to solve such DRO shortfall risk problems. 
Under proper assumptions, it will converge to the optimal value and optimal solution
of the original problem.
In addition, our algorithm can recover the probability measure that satisfies the 
worst-case expectation constraint. 
Finally, the effectiveness of the proposed model and the efficiency of the 
algorithm are represented through numerical experiments.
An interesting direction for future work involves 
identifying specialized ambiguity set structures that 
guarantee finite convergence for our algorithm.

\end{document}